\documentclass[12pt,english,a4paper]{article}

\usepackage{amssymb}
\usepackage[T1]{fontenc}
\usepackage{amsmath}
\usepackage[dvips]{graphicx}
\usepackage[symbol]{footmisc}
\usepackage{amsmath,amsfonts}
\usepackage{shadow}
\usepackage{latexsym}
\usepackage{latexsym}
\usepackage{textcomp}
\usepackage{indentfirst}
\usepackage{mathrsfs}
\usepackage{babel}

\allowdisplaybreaks

\def\AA{\mathbb{A}}
\def\GG{\mathbb{G}}

\def\Ex{\mathbf{E}}

\def\TT{\mathbb{T}}

\def\RR{\mathbb{R}}

\def\II{\mathbb{I}}

\def\DD{\mathbb{D}}

\def\Liminf{\mathop{\underline{\lim}}\limits}

\def\Pb{\mathbf{P}}

\let\bar\overline
\let\tilde\widetilde
\def\limnto{\mathrel{\mathop{\longrightarrow\kern 0pt}\limits_{n\to\infty}}}

\let\bar\overline
\def\1{\mbox{1\hspace{-.25em}I}}

\let\tilde\widetilde

\newcommand{\theoremname}{Theorem}
\newcommand{\definitionname}{Definition}
\newcommand{\propositionname}{Proposition}
\newcommand{\lemmaname}{Lemma}
\newcommand{\corollaryname}{Corollary}
\newcommand{\propertyname}{Property}
\newcommand{\exercisename}{Exercise}
\newcommand{\remarkname}{Remark}

\newcommand{\recallname}{Recall}
\newcommand{\notationname}{Notation}

\newtheorem{theorem}{\theoremname}[section]

\newcommand{\Sectionmark}{}
\newcommand{\Section}[2]{    \renewcommand{\Sectionmark}{2}  \markboth{Section \thesection\quad---\quad\textsc{\Sectionmark}}}

\addto\extrasfrench{
\providecommand{\fg}{\ifdim\lastskip>\z@\unskip\fi~\frqq}}
\makeatletter

\@addtoreset{equation}{section}
\makeatother

\begin{document}

\title{\textbf{Method of Moments Estimators and Multu-step MLE for
    Poisson Processes}}
\date{}
\author{Ali S.\ Dabye\footnotemark[1] , Alix A.\ Gounoung\footnotemark[1] ,
Yury A.\ Kutoyants \footnotemark[2]}
\maketitle

\begin{abstract}
We introduce two types  of estimators of the finite-dimensional parameters in
the case of observations of inhomogeneous Poisson processes. These are the
estimators of the method of moments and multi-step MLE. It is shown
that the estimators of the method of moments are consistent and asymptotically
normal and the multi-step MLE are consistent and asymptotically
efficient. The construction of multi-step MLE-process is done in two
steps. First we construct  a consistent estimator by the observations on some
learning interval and then this estimator is used for construction of one-step
and two-step MLEs. The main advantage of the proposed approach is its
computational simplicity.
\end{abstract}


\selectlanguage{english}

\footnotetext[1]{%
Laboratoire d'\'Etudes et de Recherches en Statistique et D\'eveloppement
(LERSTAD), Universit\'e Gaston Berger, S\'en\'egal} \footnotetext[2]{%
Le Mans University, Le Mans, France, 
Tomsk
State University, Tomsk, Russia and National Research University ``MPEI'',
Moscow, Russia }
\bigskip{}

\textbf{Key words}: Parameter estimation, inhomogeneous Poisson processes,
method of moments estimators, consistency, asymptotic normality, asymptotic
efficiency, multi-step MLE.

\bigskip{}

\textsl{AMS subject classification}: 62F10, 62F12, 62M05, 62G20.

\section{Introduction}

This work is devoted to the problem of parameter estimation in the case of
continuous time observations of inhomogeneous Poisson processes. The Poisson
process is one of the main models in the description of the series of events
in real applied problems in optical telecommunications, biology, physics,
financial mathematics etc. (see, e.g., \cite{ALZ02}, \cite{BD69}, \cite{BG03},
\cite{RB00}, \cite{S16}).  Note that the intensity function entirely
identifies the process and therefore the statistical inference is concerned
this function only. We suppose that the intensity function of the observed
Poisson process is a known function which depends on some unknown
finite-dimensional parameter.  We consider the problem of this parameter
estimation in the asymptotics of large samples.  We have to note that the
estimation theory (parametric and non parametric) is well developed and there
exists a large number of publications devoted to this class of problems (see,
e.g., \cite{D77}, \cite{Kut84}, \cite{SM91}, \cite{Kut98} and the references
therein). The method of moments and one-step estimation procedure in the case
of i.i.d. observations are well
known too. Our goal  is to apply the method of moments to the estimation of
the parameters of inhomogeneous Poisson processes and to present a version of
one-step and multi-step procedures with the help of some preliminary
estimators obtain on the small learning interval.

We are given  $n$ independent observations
$X^{\left(n\right)}=\left(X_1,\ldots,X_n\right) $ of the Poisson processes
$X_j=\left(X_j\left(t\right),t\in \TT\right)$ with the same intensity function
$\lambda \left(\vartheta ,t\right),t\in \TT$.  Here $\TT $ is an interval of
observations. It can be finite, say, $\TT=\left[0,T \right]$ or infinite
$\TT=[0,\infty )$, $\TT=(-\infty ,\infty )$. The unknown parameter $\vartheta
  \in \Theta $, where the set $\Theta $ is an open, convex and bounded subset of $ {\cal
    R}^d$.    Recall that the increments of the
  Poisson process ($X_j$ is a  counting process) on disjoint intervals are
  independent and for any   $k=0,1,2,\ldots$  and $t_1<t_2$
$$
\Pb_\vartheta \Bigl(X_j\left(t_2\right)-X_j\left(t_1\right)=k\Bigr)=\frac{\left[\int_{t_1}^{t_2}\lambda
  \left(\vartheta ,s\right){\rm d}s\right]^k
}{k!}\exp\left\{\int_{t_1}^{t_2}\lambda
  \left(\vartheta ,s\right){\rm d}s\right\}.
$$
Recall that
\begin{align*}
\Ex_\vartheta X_j\left(t\right)=\Lambda \left(\vartheta
,t\right)=\int_{}^{t}\lambda \left(\vartheta ,s\right){\rm d}s,\qquad t\in
\TT.
\end{align*}

  We have to estimate the true value of $\vartheta=\vartheta_0 $ by the
  observations $X^n$ and to describe the asymptotic ($n\rightarrow \infty $)
  properties of estimators.
 It is known that under regularity conditions  the method of
 moments estimators in the case of i.i.d. observations of the random variables are
 consistent and asymptotically normal (see, e.g. \cite{Bor98}, \cite{L99}).
Our goal is to introduce the estimators of the method of moments in the case
of observations of inhomogeneous Poisson processes. This method of estimation
was introduced by Karl Pearson in 1894 in the case of observations of the
i.i.d. random variables. Then it was extended to many other models of
observations and widely used in applied problems. It seems that till now this
method was not yet used for the estimation of the parameters of inhomogeneous
Poisson processes.

 The maximum likelihood estimator (MLE) $\hat\vartheta _n$ (under regularity
 conditions) is consistent, asymptotically normal
\begin{align*}
\sqrt{n}\left(\hat\vartheta _n-\vartheta _0\right)\Longrightarrow {\cal
  N}\left(0, \II\left(\vartheta _0\right)^{-1}\right)
\end{align*}
 and asymptotically efficient (see, e.g., \cite{Kut84}). Here
 $\II\left(\vartheta _0\right)$ is the Fisher information matrix
\begin{align*}
\II\left(\vartheta _0\right)=\int_{\TT}^{}\dot \lambda \left(\vartheta
_0,t\right)\dot \lambda \left(\vartheta _0,t\right)^\tau  \lambda \left(\vartheta
_0,t\right)^{-1} {\rm d}t .
\end{align*}
Here and in the sequel dot means derivation w.r.t. $\vartheta $ and $A^\tau $
means the transpose of the vector (or matrix) $A$.

Recall that in the regular case the following lower bound (called Hajek-Le Cam)
holds: for any estimator $(\bar\vartheta _n $ and any $\vartheta _0\in\Theta $
we have
\begin{align*}
\lim_{\nu \rightarrow 0 }\Liminf_{n\rightarrow \infty }\sup_{\left|\vartheta
  -\vartheta _0\right|<\nu }n\Ex_\vartheta \left| \II\left(\vartheta
_0\right)^{1/2}\left(\bar\vartheta
_n-\vartheta\right) \right|^2 \geq d.
\end{align*}
Thois bound allows us to define  the asymptotically efficient estimator
$\check\vartheta _n$ as estimator satisfying the equality
\begin{align*}
\lim_{\nu \rightarrow 0 }\Liminf_{n\rightarrow \infty }\sup_{\left|\vartheta
  -\vartheta _0\right|<\nu }n\Ex_\vartheta \left| \II\left(\vartheta
_0\right)^{1/2}\left(\check\vartheta
_n-\vartheta\right) \right|^2=  d
\end{align*}
for all $\vartheta _0\in\Theta $.

If we verify that the moments of the MLE converge uniformly on $\vartheta $
then this proves the asymptotic efficiency of the MLE (see \cite{IH81}, \cite{Kut84}).

In the present work we introduce two classes of estimators. The first one is
the class of the method of moments estimators (MME) and the second class is
the multi-step MLEs.

We show that the MMEs for many models of inhomogeneous Poisson processes are
easy to calculate, but these estimators as usual are not asymptotically
efficient. The MLEs are asymptotically efficient, but their calculation is
often a difficult problem.  The main result of this work is the introduction
of the multi-step MLEs which are easy to calculate and which are
asymptotically efficient. These multi-step MLEs are calculated in several
steps. For example, one-step MLE is calculated as follows. First we fix the
{\it learning observations} $X^{N}=\left(X_1,\ldots,X_N\right)$, where
$N=\left[n^\delta \right]$ with $\delta \in \left(\frac{1}{2},1\right)$. Here
$\left[a\right]$ is the entier part of $a$. By the observations $X^{N}$ we
construct the MME $\vartheta _N^*$ and then with the help of it we introduce
the one-step MLE by the equality
\begin{align*}
\vartheta _n^\star=\vartheta _N^*+\frac{1}{\sqrt{n}}\II\left(\vartheta
_N^*\right)^{-1}\sum_{j=N+1}^{n}\int_{\TT}^{}\dot \lambda \left(\vartheta
_N^*,t\right)\lambda \left(\vartheta _N^*,t\right)^{-1}\left[{\rm
    d}X_j\left(t\right)-\lambda \left(\vartheta _N^*,t\right){\rm d}t\right] .
\end{align*}
It is shown that this estimator is asymptotically normal
\begin{align*}
\sqrt{n}\left(\vartheta _n^\star-\vartheta _0\right)\Longrightarrow {\cal
  N}\left(0, \II\left(\vartheta _0\right)^{-1}\right)
\end{align*}
and is asymptotically efficient.

 Recall that the MLE can be explicitly written for the very narrow class of
 intensities. Therefore it is important to have other estimators, which are
 consistent and asymptotically normal and the same time can be easily
 calculated.

\section{ Method of Moments for Poisson processes}

Let us construct  the method of moments estimator in the case of
observations of inhomogeneous Poisson process.  We have $n$
independent observations $X^{n}=(X_{1},X_{2},...,X_{n}) $ of the Poisson
processes  $X_{j}=(X_{j}(t), t\in \TT ) $ with the intensity function
$(\lambda \left( \theta ,t\right) , t\in \TT  )$.

The  unknown parameter $\theta \in \Theta \subset \mathbb{R}%
^{d}$. Here $\Theta $ is an open, convex, bounded set.

In the construction of the   method of moments estimator (MME) we follow
the same way as in the construction of MME in the case of i.i.d. random
variables.
Introduce the vector-function ${\bf g}\left( s\right) =(g_{1}\left( s\right)
,...,g_{d}\left( s\right) ), t\in \TT $ and the vector of integrals
${\bf I}^{\left(d\right)}=\left(I_{1}
,\ldots, I_{d}\right)$, where
$$
I_{l} =\int_{\TT}^{ }g_{l}\left( s\right)
{\rm d}X_{1}\left( s\right),\qquad l=1,\ldots,d.
$$
We have
$$
\mathbb{E}_{\theta }{\bf I}^{\left(d\right)} =\int_{\TT}^{ }{\bf g}\left(
s\right) \lambda \left( \theta ,s\right) {\rm d}s.
$$
Let us denote ${\bf M}\left( \vartheta  \right) =\mathbb{E}_{\theta }{\bf
  I}^{\left(d\right)}$ and suppose that the function ${\bf g}\left( \cdot
\right) $ is such that the equation
$
{\bf M}\left( \vartheta  \right) ={\bf a}
$
for all $\vartheta \in \Theta $ has a unique solution
$
\vartheta={\bf M}^{-1}\left( {\bf a}  \right) ={\bf H}\left( {\bf a}  \right).
$
Here ${\bf H}\left( {\bf a}  \right) $ is the inverse function for ${\bf
  M}\left( \cdot  \right) $.

The   method of moments  estimator  $\vartheta _n^*$ is defined by the
equation
$$
\vartheta _n^*={\bf H}\left( {\bf a}_n  \right),
$$
where
$$
{\bf a}_n=\frac{1}{n}\sum_{j=1}^{n}\int_{\TT}^{ } {\bf g}\left( s\right){\rm d}X_j\left(s\right)
$$

Introduce the {\it Regularity conditions  ${\cal R}_0$ :}
\begin{itemize}
\item {\it For any $\nu >0$ and any} $\vartheta_0 \in\Theta $
$$
\inf_{\left|\vartheta -\vartheta _0\right|>\nu }\left|{\bf M}\left( \vartheta
\right)-{\bf M}\left( \vartheta_0  \right) \right| >0.
$$
\item {\it The vector-function ${\bf H}\left( \cdot  \right) $ is continuously
  differentiable. }
\end{itemize}
Introduce the matrix
$$
\DD\left(\vartheta\right)=\frac{\partial {\bf H}\left( \vartheta
  \right)}{\partial \vartheta }\GG\left(\vartheta \right) \frac{\partial {\bf H}\left( \vartheta
  \right)}{\partial \vartheta }^T.
$$
Here the matrices
\begin{align*}
\left(\frac{\partial {\bf H}\left( \vartheta
  \right)}{\partial \vartheta }\right)_{lk}=\frac{\partial { H}_l\left( \vartheta
  \right)}{\partial \vartheta_k },\qquad \GG\left(\vartheta
\right)_{l,k}=\int_{\TT}^{} g_l\left(s\right)g_k\left(s\right)\lambda
\left(\vartheta ,s\right){\rm d}s.
\end{align*}

\begin{theorem}
\label{T1}
Suppose that the vector-function ${\bf g}\left( \cdot \right) $ is such, that the
regularity conditions ${\cal R}_0$ are fulfilled. Then the MME $\vartheta _n^*$
is consistent and asymptotically normal
\begin{equation}
\label{01}
\sqrt{n}\left(\vartheta _n^*-\vartheta _0 \right)\Longrightarrow {\cal
  N}\left(0, \DD\left(\vartheta _0\right)\right).
\end{equation}

\end{theorem}
{\bf Proof.}  By the Law of Large Numbers
\begin{equation*}
a_{l,n}=\frac{1}{n}\sum_{j=1}^{n}\int_{\TT}^{ }g_l\left(s\right){\rm
  d}X_j\left(s\right)\longrightarrow \int_{\TT}^{ }g_{l}\left( s\right)
\lambda \left( \vartheta_0 ,s\right) {\rm d}s, \quad l=1,...,d
\end{equation*}%
and hence by Continuous Mapping Theorem
${\bf H}\left( {\bf a}_{n}\right) \longrightarrow {\bf H}\left( {\bf a}_0\right)
=\vartheta_0  .$
Here we put ${\bf a}_0={\bf M}\left(\vartheta _0\right) $.
To show asymptotic normality we write
\begin{align*}
\sqrt{n}\left( \vartheta _n^*-\vartheta _0\right) =\sqrt{n}\left({\bf H}\left(
     {\bf a}_{n}\right) -{\bf H}\left( {\bf a}_0\right)\right)=\sqrt{n}\left({\bf H}\left(
     {\bf a}_{ 0}+b_n\eta _n\right) -{\bf H}\left( {\bf a}_0\right)\right)
\end{align*}
where $b_n=n^{-1/2}$ and the vector
\begin{align*}
\eta _n=\frac{1}{\sqrt{n}}\sum_{j=1}^{n}\int_{\TT}^{ }{\bf g}\left(s\right)\left[{\rm d}X_j\left(s\right)-\lambda \left(\vartheta _0,s\right){\rm d}s\right].
\end{align*}
By the Central Limit Theorem
\begin{align*}
\eta _n\Longrightarrow {\cal N}\left(0,\GG\left(\vartheta \right)\right).
\end{align*}
The asymptotic normality  \eqref{01} now follows from this convergence and the
presentation
\begin{align*}
\sqrt{n}\left( \vartheta _n^*-\vartheta _0\right) =\frac{\partial {\bf H}\left( \vartheta
  \right)}{\partial \vartheta }\eta _n\left(1+o\left(1\right)\right).
\end{align*}
Recall that the vector-function $H\left(a\right) $ is continuously
differentiable.

\bigskip

{\bf Example 1}.
Suppose that the  intensity function is
\begin{equation*}
\lambda \left( \theta ,t\right) =\sum_{l=1}^{d}\theta _{l}h_{l}\left(
t\right) +\lambda _{0},  \qquad t\in\TT.
\end{equation*}
Introduce the vector-function ${\bf g}\left( \cdot \right) $ and the corresponding
integrals  ${\bf I}^{\left(d\right)} $. The vector
${\bf M}\left(\vartheta \right)=\AA \vartheta +\lambda _0{\bf G}$,
where
\begin{eqnarray*}
\AA_{kl} =\int_{\TT}^{ }g_k\left(t\right) h_l\left(t\right){\rm d}t,\qquad
G_{k}=\int_{\TT}^{ }g_k\left(t\right) {\rm d}t
\end{eqnarray*}
in obvious notations. Hence we can write
$$
\vartheta =\AA^{-1}\left[{\bf M}\left(\vartheta \right)-\lambda _0{\bf G}
  \right]=\AA^{-1}\left[{\bf a}-\lambda _0{\bf G}
  \right] ={\bf H}\left( {\bf a}\right).
$$

Therefore
the MME $\vartheta_{n}^{* } $ is given by the equality
\begin{equation}
\label{mme1}
\vartheta _{n}^{* }=\AA^{-1}\left[ {\bf a}_{n}-\lambda _{0}{\bf
    G}\right]=\AA^{-1}\frac{1}{n }\sum_{j=1}^{n}\int_{\TT}^{} {\bf
  g}\left(t\right)\left[{\rm d}X_j\left(t\right)-\lambda _0{\rm d}t\right].
\end{equation}
This estimator by the Theorem \ref{T1} is consistent and asymptotically
normal. To simplify its calculation we can take such functions $ {\bf
  g}\left(\cdot \right)$  that the matrix $ \AA$ became diagonal.

{\bf Example 2}.
Suppose that the inhomogeneous Poisson processes $X^{\left(n\right)}$  are
observed on the time interval $\TT=[0,\infty )$ and have the intensity function
\begin{equation*}
\lambda \left( \vartheta  ,t\right) =\frac{t^{\beta -1}\alpha ^{\beta }}{\Gamma
\left( \beta \right) }\exp \left( -\alpha t\right) , t\geq 0,
\end{equation*}
i.e., we have  Poisson processes with the Gamma intensity function. The unknown
parameter is $\vartheta  =\left( \alpha ,\beta\right) .$
We know, that
\begin{align*}
M_{1}\left( \vartheta  \right)  =\int_{0}^{\infty }t\mathbf{\ }\lambda \left(
\vartheta  ,t\right) {\rm d}t=\frac{\beta }{\alpha }, \qquad
M_{2}\left( \vartheta  \right)  =\int_{0}^{\infty }t^{2}\lambda
\left( \vartheta  ,t\right) {\rm d}t=\frac{\beta\left( \beta +1\right)  }{\alpha ^{2}}.
\end{align*}
Hence, if we take $g\left( t\right) =\left(g_{1}\left( t\right) ,g_{2}\left(
t\right) \right)=\left( t,t^{2}\right) ,$ then the system ${\bf M}\left(\vartheta \right)= {\bf a}$
has the unique solution
\begin{equation*}
          \alpha  =\frac{a_{1}}{a_{2}-a_{1}^{2}},\qquad
       \beta  =\frac{a_{1}^{2}}{a_{2}-a_{1}^{2}}.
      \end{equation*}
Therefore the MME $\vartheta _n^*=\left(\alpha _{n}^{\ast },\beta _{n}^{\ast } \right)$ is
\begin{align}
\label{mme2}
\alpha _{n}^{\ast } &=\frac{\frac{1}{n}\sum_{j=1}^{n}\!\int_{0}^{\infty
}t{\rm d}X_{j}\left( t\right)}{  \left(
\frac{1}{n}\sum_{j=1}^{n}\!\int_{0}^{\infty  }t^{2}{\rm d}X_{j}\left( t\right)
\!-\!\left( \frac{1}{n}\sum_{j=1}^{n}\!\int_{0}^{\infty  }t{\rm d}X_{j}\left( t\right)
\right)^{2}
\right)}, \\
\label{mme3}
\beta _{n}^{\ast } &=\frac{\left( \frac{1}{n}\sum_{j=1}^{n}\!\int_{0}^{\infty
}t{\rm d}X_{j}\left( t\right) \right)^{2}}{\left(
\frac{1}{n}\sum_{j=1}^{n}\!\int_{0}^{\infty  }t^{2}{\rm d}X_{j}\left( t\right)
\!-\!\left( \frac{1}{n}\sum_{j=1}^{n}\!\int_{0}^{\infty  }t{\rm d}X_{j}\left( t\right)
\right)^{2}
\right)}.
\end{align}
 This estimator is consistent and asymptotically normal.

The similar example can be considered and in the case of observations on
$\TT=\left(-\infty ,+\infty \right)$ and the Gaussian intensity function   with $\vartheta =\left(\alpha ,\sigma ^2\right)$:
\begin{equation*}
\lambda \left( \vartheta  ,t\right) =\frac{1}{\sqrt{2\pi \sigma ^{2}}}\exp
\left\{ -\frac{\left( t-\alpha \right) ^{2}}{2\sigma ^{2}}\right\} , \ t\in
\mathbb{R}.
\end{equation*}%

\section{ One-Step MLE}

 The one-step MLE was introduced by Fisher (1925). This one-step
 procedure allows  to improved a
 consistent estimator $\bar\vartheta _n$ up to asymptotically efficient
 (one-step MLE)  $\vartheta _n^\star$.
We consider the similar construction  in the case of inhomogeneous Poisson
processes. Suppose that the observations $X^{\left(n\right)}=\left(
X_{1},...,X_{n}\right) $ are Poisson processes with the intensity function
$ \lambda (\vartheta ,t),t\in \TT$.

{\it Condition ${\cal P}_0$}.  {\it We have a (preliminary) estimator $\bar{\vartheta }%
_{n} $, which is consistent and such that $\sqrt{n}
\left( \bar{\vartheta  } _{n}-\vartheta _{0}\right) $ is bounded in
probability. }

 Introduce the learning
observations $X^{\left(N\right)}=\left(X_{1},...,X_{N}\right)$ and   the one-step MLE
\begin{equation*}
\vartheta _{n}^{\star }= \bar{\vartheta }_{N}+\frac{{\II\left( \bar{
\vartheta }_{N}\right) ^{-1}}}{{n}}\sum_{j=N+1}^{n}\int_{\TT}^{ }\frac{%
\dot{\lambda }\left( \bar{\vartheta }_{N},t\right) }{\lambda
\left( \bar{\vartheta }_{N},t\right) }\left[ {\rm d}X_{j}\left( t\right)
-\lambda \left( \bar{\vartheta }_{N},t\right) {\rm d}t\right] .
\end{equation*}
Here $\bar{\vartheta }_{N} $ is the preliminary estimator constructed by
the first $N$ observations.

{\it Regularity conditions ${\cal L}_0$}:
\begin{itemize}
\item {\it The function $l\left( \vartheta ,t\right) =\ln \lambda \left(
\vartheta ,t\right) $ has three continuous bounded derivatives.
w.r.t. $\vartheta $
\item  The Fisher
information matrix $\II\left(\vartheta \right)$
is uniformly on $\vartheta \in\Theta $ non degenerated:
\begin{equation*}
\inf_{\vartheta \in\Theta }  \inf_{\left|\mu\right| =1} \mu ^\tau\II\left(\vartheta \right)\mu  >0.
\end{equation*}
}
\end{itemize}
Here $\mu \in \RR^d$.

\begin{theorem}
\label{T2}
Suppose that the conditions ${\cal P}_0$ and ${\cal L}_0$ are fulfilled.
Then the one-step MLE $\vartheta _{n}^{\star }$ is asymptotically
normal%
\begin{equation*}
\sqrt{n}\left( \vartheta _{n}^{\star }-\vartheta _{0}\right) \Longrightarrow
\mathcal{N}\left( 0,\II\left(\vartheta_0 \right) ^{-1}\right) .
\end{equation*}
\end{theorem}
\begin{proof}
We have the equality%
\begin{align*}
\sqrt{n}\left( \vartheta _{n}^{\star }-\vartheta _{0}\right)  &=\sqrt{n}\left(
\bar{\vartheta }_{N}-\vartheta _{0}\right) + \\
&+{\II\left( \bar\vartheta_{N}\right)^{-1} }\frac{1}{\sqrt{n}}
\sum_{j=N+1}^{n}\int_{\TT}^{ }\dot{\ell}\left( \bar{\vartheta }
_{N},t\right) \left[ {\rm d}X_{j}\left( t\right) -\lambda \left( \vartheta
_{0},t\right) {\rm d}t\right] + \\
&+\II\left( \bar\vartheta_{N}\right)^{-1}    \frac{n-N}{\sqrt{n}}
\int_{\TT}^{ }\dot{\ell}\left( \bar{\vartheta }_{N},t\right)
\left[ \lambda \left( \vartheta _{0},t\right) -\lambda \left( \bar{\vartheta }
_{N},t\right) \right] {\rm d}t.
\end{align*}

As $\bar{\vartheta }_{N}\longrightarrow \vartheta _{0} $ we
 can write%
\begin{align*}
&\II\left( \bar\vartheta_{N}\right)^{-1}\frac{1}{\sqrt{n}}
\sum_{j=N+1}^{n}\int_{\TT}^{ }\dot{\ell}\left( \bar{\vartheta }
_{N},t\right) \left[ {\rm d}X_{j}\left( t\right) -\lambda \left( \vartheta
_{0},t\right) {\rm d}t\right]  \\
&\qquad =\II\left( \vartheta_{0}\right)^{-1}\frac{1}{\sqrt{n}}
\sum_{j=N+1}^{n}\int_{\TT}^{ }\dot{\ell}\left( \vartheta _{0},t\right)
\left[ {\rm d}X_{j}\left( t\right) -\lambda \left( \vartheta _{0},t\right) {\rm d}t\right]
+o\left( 1\right) .
\end{align*}

By the Central Limit Theorem
\begin{equation*}
\II\left( \vartheta_{0}\right)^{-1}\frac{1}{\sqrt{n}}\sum_{j=N+1}^{n}%
\int_{\TT}^{ }\dot{\ell}\left( \vartheta _{0},t\right) \left[
{\rm d}X_{j}\left( t\right) -\lambda \left( \vartheta _{0},t\right) {\rm d}t\right]
\Longrightarrow \mathcal{N}\left( 0,\II\left( \vartheta_{0}\right)^{-1}\right) .
\end{equation*}

Let us consider the remainder
\begin{align*}
R_{n} &=\sqrt{n}\left( \bar{\vartheta }_{N}-\vartheta _{0}\right)\\
 &\quad
 +\II\left( \bar\vartheta_{N}\right)^{-1}
\frac{n-N}{\sqrt{n}}\int_{\TT}^{} \dot{\ell}\left( \bar{\vartheta
}_{N},t\right) \left[ \lambda \left( \vartheta _{0},t\right) -\lambda \left(
  \bar{\vartheta }_{N},t\right) \right] {\rm d}t \\
&=\sqrt{n}\left(
\bar{\vartheta }_{N}-\vartheta _{0}\right)\II\left(
\bar\vartheta_{N}\right)^{-1} \left[ \II\left( \bar\vartheta_{N}\right)
  -\int_{\TT}^{ }\dot{\lambda }( \bar{\vartheta }_{N},t)^\tau\dot{\ell}\left( \bar{\vartheta }_{N},t\right)
  {\rm d}t\right]\\
&\quad  + \sqrt{n}\left(
\bar{\vartheta }_{N}-\vartheta _{0}\right) O\left(\frac{N}{n}\right)+ O\left( \sqrt{n} \left( \bar{\vartheta
}_{N}-\vartheta _{0}\right)^2\right)=o\left(1\right),
\end{align*}
where we used the equality
\begin{align*}
\II\left( \bar\vartheta_{N}\right)
  =\int_{\TT}^{ }\dot{\lambda }( \bar{\vartheta }_{N},t)^\tau\dot{\ell}\left( \bar{\vartheta }_{N},t\right)
  {\rm d}t
\end{align*}
and the Taylor expansion at the point $\bar{\vartheta }_{N} $:
\begin{align*}
\lambda \left( \vartheta _{0},t\right) -\lambda \left( \bar{\vartheta }%
_{N},t\right) &=-
\int_{0}^{1}\dot{\lambda }\left( {\bar\vartheta_N+s\left( \bar{\vartheta
}_{N}-\vartheta _{0}\right)  },t\right)^\tau\left( \bar{\vartheta
}_{N}-\vartheta _{0}\right){\rm d}s\\
&= -
\dot{\lambda }\left( \bar{\vartheta
}_{N},t\right)^\tau\left( \bar{\vartheta
}_{N}-\vartheta _{0}\right)+O\left(\left( \bar{\vartheta
}_{N}-\vartheta _{0}\right)^2\right) .
\end{align*}
Therefore we obtained the representation %
\begin{eqnarray*}
\sqrt{n}\left( \vartheta _{n}^{\star }-\vartheta _{0}\right) &=&\II\left(
\vartheta_{0}\right)^{-1}\frac{1}{\sqrt{n}}\sum_{j=N+1}^{n}\int_{\TT}^{ }
\dot{\ell}\left( \vartheta _{0},t\right) \left[ {\rm d}X_{j}\left( t\right)
  -\lambda \left( \vartheta _{0},t\right) {\rm d}t \right] +o\left( 1\right)
\end{eqnarray*}
 which proves the theorem.
\end{proof}

{\bf Remark 1}. If we suppose that the moments of the preliminary estimator
are bounded, say,
\begin{align*}
\Ex_{\vartheta_0} \left|\bar\vartheta _n-\vartheta _0\right|^p\leq C
\end{align*}
where $p\geq 2$ and $C>0$ does not depend on $n$, then the presented proof
allows to verify that the moments of the one-step MLE are bounded too and that
$\vartheta _n^\star$ is asymptotically efficient.

 In all examples below the MLEs have no explicit expression.

{\bf Example 1}.
 Suppose that the intensity function is
$$\lambda \left( \vartheta ,t\right) =\sum_{l=1}^{d}\vartheta_l h_l\left( t\right) +\lambda
 _{0}, t\in \TT$$ and $\vartheta _n^*$ is the MME defined in \eqref{mme1}.

The Fisher information matrix is
\begin{align*}
\II\left(\vartheta
\right)_{lk}=\int_{\TT}^{}\frac{h_l\left(t\right)h_k\left(t\right)}{
  h\left( t\right)^\tau\vartheta +\lambda  _{0} }\;{\rm d}t ,\qquad
l,k=1,\ldots,d
\end{align*}
 and the one-step MLE in this case is
\begin{align*}
\vartheta _{n}^{\star } &=\vartheta_{N}^*+\II\left(\vartheta_{N}^*
\right)^{-1}\frac{1}{n}\sum_{j=N+1}^{n}\int_{\TT}^{ }\frac{h\left( t\right)
}{h\left( t\right)^\tau\vartheta_{N}^* +\lambda _{0}} \left[ {\rm
    d}X_{j}\left( t\right) -h\left( t\right)^\tau{\vartheta }_{n}^* {\rm
    d}t-\lambda _{0}{\rm d}t\right].
\end{align*}%
Here $N=\left[n^\delta \right]$ and $\delta \in
\left(\frac{1}{2},1\right)$. By the Theorem \ref{T2} this estimator is
consistent and asymptotically normal.  Therefore we improved the preliminary
estimator ${\vartheta }_{N}^*$ up to asymptotically efficient $\vartheta
_{n}^{\star }$ .

{\bf Example 2}.
Suppose that the intensity function is
\begin{equation*}
\lambda \left( \vartheta ,t\right) =\frac{t^{\beta -1}\alpha ^{\beta }\exp
\left( -\alpha t\right) }{\Gamma \left( \beta \right) },\qquad t\geq 0,
\end{equation*}%
where the unknown parameter is $\vartheta =\left(\alpha ,\beta\right) $. Once
more we have a situation, where the explicit calculation of the MLE is
impossible.  The preliminary estimator can be the MME ${\vartheta
}_{n}^*=\left(\alpha _n^*,\beta _n^*\right)$ (see \eqref{mme2} and
\eqref{mme3}).

The vector $\dot l \left(\vartheta ,t\right)=\left(\frac{\beta }{\alpha }-t,
\ln\left(\alpha t\right) -\frac{\dot \Gamma \left(\beta \right)}{\Gamma
  \left(\beta \right)}  \right)$ and the Fisher information matrix $\II\left(\vartheta
\right)=\left(\II_{lk}\left(\vartheta \right) \right)_{2\times 2}$  is
\begin{eqnarray*}
\II_{11}\left(\vartheta \right)=\frac{\beta }{\alpha ^2},\qquad
\II_{12}\left(\vartheta \right)=-\frac{1}{\alpha },   \quad
\II_{22}\left(\vartheta \right)= \frac{\ddot \Gamma \left(\beta \right)\Gamma
  \left(\beta \right)-\dot \Gamma \left(\beta \right)^2 }{\Gamma \left(\beta
  \right)^2}.
\end{eqnarray*}

Hence the one-step MLE is
\begin{align*}
\vartheta _{n}^{\star } &=\vartheta_{N}^*+\II\left(\vartheta_{N}^*
\right)^{-1}\frac{1}{n}\sum_{j=N+1}^{n}\int_{\TT}^{ } \dot l \left({\vartheta
}_{N}^* ,t\right)\left[ {\rm d}X_{j}\left( t\right) -\lambda \left( {\vartheta
  }_{N}^*,t\right) {\rm d}t\right]
\end{align*}
and this estimators is asymptotically normal with the limit covariance matrix
$\II\left(\vartheta _0\right)^{-1}$.

\section{One-step MLE-process}

Suppose that we have the same model of observations of $n$ independent
inhomogeneous Poisson processes: $X^{n}=\left( X_{1},..., X_{n}\right) $ with
the intensity function $\lambda(\vartheta,t), t\in \TT $, where $ \vartheta$
is unknown parameter. Our goal is to construct an estimator
process $ \vartheta ^\star_n=\left( \vartheta
^\star_{k,n},k=1,\ldots,n\right)$, where the estimator $\vartheta^\star_{k,n}
$ satisfies the following conditions
\begin{enumerate}
\item {\it The estimator $\vartheta
^\star_{k,n} $  is based on the first $k$ observations
$X^{\left(k\right)}$.
\item  The calculation of this estimator has to be relatively
simple.

\item The   estimator $\vartheta^\star_{k,n} $ is asymptotically efficient. }
\end{enumerate}

Note that the MLE $\hat{\vartheta }_{k,n}$ defined by the  relations
\begin{equation}
\label{mleq}
V\left(\hat{\vartheta }_{k,n}, X^k\right)=\sup_{\vartheta \in \Theta}
V\left(\vartheta, X^k\right),\qquad  k=1,...,n
\end{equation}%
satisfies the conditions (1) and (3), but not (2).  is The likelihood ratio
function \cite{LS05} $V\left(\vartheta,
X^k\right),\vartheta \in\Theta  $ is
\begin{align*}
V\left(\vartheta, X^k\right)=\exp\left\{\sum_{j=1}^{k} \int_{\TT}^{} \ln
\lambda \left(\vartheta ,t\right){\rm
  d}X_j\left(t\right)-k\int_{\TT}^{}\left[\lambda \left(\vartheta ,t\right)-1
  \right]{\rm d}t\right\} .
\end{align*}
Remind that the solutions of the
equations \eqref{mleq} in the case of non linear intensity functions
$\lambda(\vartheta,\cdot )$ can be  computationally  difficult problems.
This is typical situation of "on-line" estimation.

The construction of such estimator-process is very close to the given above
construction of the
One-step MLE.  Introduce the same  learning observations $X^{N}=\left( X_{1},...,
X_{N}\right) $, where $N=[n^{\delta}]$, with $\delta \in \left(\frac{1}{2},
1\right)$ and suppose that we have a preliminary estimator $ \bar\vartheta _N$
such that $\sqrt{N}\left( \bar\vartheta _N-\vartheta _0\right)$ is bounded in
probability (condition ${\cal P}_0$).

The One-step MLE-process is
\begin{align*}
\vartheta _{k,n}^{\star }=\bar{\vartheta }_{N} + \II\left(\bar{\vartheta }_{N}
\right)^{-1}\frac{1}{k}\sum_{j=N+1}^{k}\int_{\TT}^{}\dot{\ell }%
\left( \bar{\vartheta }_{N} ,t\right)\left[ {\rm d}X_{j}\left( t\right) -\lambda
\left( \bar{\vartheta }_{N} ,t \right) {\rm d}t\right],
\end{align*}
where $k=N+1,...,n.$

\begin{theorem}
\label{T3}
Suppose that the conditions ${\cal P}_0$ and ${\cal L}_0$ are fulfilled.  Then the
One-step MLE-process $\vartheta_n^{\star}=\left( \vartheta _{k,n}^{\ast },
k=N+1,...,n\right)$ is consistent and asymptotically normal
\begin{equation*}
\sqrt{k}\left(\vartheta _{k,n}^{\star }-\vartheta_0 \right) \Longrightarrow
\mathcal{N}\left( 0,\II\left( \vartheta_0 \right) ^{-1}\right)
\end{equation*}
where we put $k=[sn] $. Here $s\in (0,1]$.
\end{theorem}

{\bf Proof}. There is no need to present a new proof because it is a slight
  modification of the given above proof of the Theorem \ref{T2}.

\section{Two-step MLE-process}

The One step MLE-process presented in the preceding section allows us to
calculate the values  $\vartheta _{k,n}^{\star }$ for $k=N+1,...,n$, where $N=%
\left[n^{\delta}\right]$ with $\delta \in ( \frac{1}{2},1] $. Therefore
we have no estimators for $k=1,...,N$.

It is interesting to reduce the learning interval and to start the estimation
process earlier. Let us see how it can be done with the learning interval
$X^N=\left(X_1,\ldots,X_N\right)$ with  $N=\left[n^{\delta}%
\right]$ and $\delta \in \Bigl(  \frac{1}{3}, \frac{1}{2}\Bigr] $.

We suppose that a preliminary estimator $\bar{\vartheta }_{N}$ is given. Then we define the second
preliminary estimator
\begin{equation*}
\bar{\vartheta }_{k,n}=\bar{\vartheta }_{N}+\II\left(\bar{
\vartheta }_{N} \right)^{-1}\frac{1}{k}\sum_{j=N+1}^{k}\int_{\TT}^{}\dot{\ell }%
\left( \bar{\vartheta }_{N} ,t\right)\left[ {\rm d}X_{j}\left( t\right) -\lambda
\left( \bar{\vartheta }_{N} ,t \right) {\rm d}t\right],
\end{equation*}
and the Two-step MLE-process is defined by the relation
\begin{equation*}
\vartheta _{k,n}^{\star \star}=\bar{\vartheta }_{k,n}+\II\left(\bar{
\vartheta }_{N} \right)^{-1}\frac{1}{k}\sum_{j=N+1}^{k}\int_{\TT}^{}\dot%
{\ell }\left( \bar{\vartheta }_{N} ,t\right)\left[ {\rm d}X_{j}\left( t\right)
-\lambda \left( \bar{\vartheta }_{k,n} ,t \right) {\rm d}t\right],
\end{equation*}
where $k=N+1,...,n.$
Let us show that it is asymptotically normal
\begin{equation*}
\sqrt{k}\left(\vartheta _{k,n}^{{\star \star}}-\vartheta_0 \right) \Longrightarrow
\mathcal{N}\left( 0,\II\left( \vartheta_0 \right) ^{-1}\right).
\end{equation*}
Here $k=\left[sn\right]$ and $
s\in ( 0,1 ] $.
 We have
\begin{align*}
\sqrt{k}\left(\vartheta _{k,n}^{\star \star}-\vartheta_0 \right)&= \sqrt{k}\left(%
\bar{\vartheta }_{k,n}-\vartheta_0 \right) +\\
&\quad +\II\left(\bar{
\vartheta }_{N} \right)^{-1} \frac{1}{k} \sum_{j=N+1}^{k}\int_{\TT}^{}\dot{\ell }%
\left( \bar{\vartheta }_{N} ,t\right)  \left[ {\rm d}X_{j}\left( t\right) -\lambda
\left( \vartheta_{0} ,t\right) {\rm d}t\right] \\
&\quad+ \II\left(\bar{
\vartheta }_{N} \right)^{-1} \frac{\left(k-N\right)}{k}%
\int_{\TT}^{}\dot{\ell }\left( \bar{\vartheta }_{N} ,t\right) %
\left[\lambda \left(\vartheta_0 ,t \right) -\lambda \left( \bar{\vartheta }%
_{k,n} ,t \right) \right]{\rm d}t.
\end{align*}
We can write for some $\gamma>0$, which we chose later
\begin{align*}
&n^{\gamma}\left(\bar{\vartheta }_{k,n}-\vartheta_0 \right)
= n^{\gamma}\left(
\bar{\vartheta }_{N}-\vartheta_0 \right) \\
&\qquad +\II\left(\bar{
\vartheta }_{N} \right)^{-1} \frac{n^{\gamma}}{k} \sum_{j=N+1}^{k}\int_{\TT}^{}\dot{\ell }
\left( \bar{\vartheta }_{N} ,t\right)\left[ {\rm d}X_{j}\left( t\right) -\lambda \left( \vartheta_{0} ,t
\right) {\rm d}t\right]\\
&\qquad +\II\left(\bar{
\vartheta }_{N} \right)^{-1}
\frac{n^{\gamma}\left(k-N\right)}{k}\int_{\TT}^{}\dot{\ell }\left( \bar{\vartheta }
_{N} ,t\right)\left[\lambda \left(\vartheta_0 ,t \right) -\lambda \left(
\bar{\vartheta }_{N} ,t \right) \right]{\rm d}t \\
&=
{n^{\gamma} \left(\bar{\vartheta }_{N}-\vartheta_0 \right)}{}\left[J -
\left(1-\frac{N}{k}\right)\II\left(\bar{\vartheta }_{N}\right)^{-1}\int_{\TT}^{}\dot{\ell }\left(
\bar{\vartheta }_{N} ,t\right)\lambda ( \tilde{\vartheta } ,t ) {\rm d}t%
\right] \\
&\qquad +\II\left(\bar{\vartheta }_{N}\right)^{-1}\frac{n^{\gamma}}{k}%
\sum_{j=N+1}^{k}\int_{\TT}^{}\dot{\ell }\left( \bar{\vartheta
}_{N} ,t\right) \left[ {\rm d}X_{j}\left( t\right) -\lambda \left( \vartheta_{0} ,t
\right) {\rm d}t\right] \\
&=O\left({n^{\gamma} \left|\bar{\vartheta }_{N}-\vartheta_0
  \right|^2}\right)+O\left(\frac{N}{k}\right)  \\
&\quad+\II\left(\bar{\vartheta }_{N}\right)^{-1}\frac{n^{\gamma}}{k}%
\sum_{j=N+1}^{k}\int_{\TT}^{}\dot{\ell }\left( \bar{\vartheta
}_{N} ,t\right) \left[ {\rm d}X_{j}\left( t\right) -\lambda \left( \vartheta_{0} ,t
\right) {\rm d}t\right].
\end{align*}%

If we take $\gamma < \delta $ then we have
\begin{equation*}
n^{\gamma}n^{-\delta}\left(n^{\frac{\delta}{2}}\left|\bar{\vartheta }_{N}-\vartheta_0
\right)\right|^2 \longrightarrow 0.
\end{equation*}
Further, as $\gamma < \delta \leq \frac{1}{2}$ we have
\begin{eqnarray*}
&&\frac{n^{\gamma}}{k}\sum_{j=N+1}^{k}\int_{\TT}^{}\dot{\ell }%
\left( \bar{\vartheta }_{N} ,t\right) \left[ {\rm d}X_{j}\left( t\right)
-\lambda \left( \vartheta_{0} ,t \right) {\rm d}t\right] \\
&&\qquad =\frac{n^{{\gamma}-\frac{1}{2}}}{\sqrt{sk}}\sum_{j=N+1}^{k}\int_{\TT}^{}\dot{\ell }
\left( \bar{\vartheta }_{N} ,t\right) \left[ {\rm d}X_{j}\left(
t\right) -\lambda \left( \vartheta_{0} ,t \right) {\rm d}t\right]  =o\left(n^{{\gamma}-\frac{1}{2}}\right) \rightarrow 0.
\end{eqnarray*}
Hence for $\gamma < \delta$
\begin{equation*}
n^{\gamma}\left(\bar{\vartheta }_{k,n}-\vartheta_0 \right)\longrightarrow 0.
\end{equation*}
Therefore
\begin{align*}
\sqrt{k}\left(\vartheta _{k,n}^{\star \star}-\vartheta_0 \right)&=O\left(\sqrt{k}\left|
\bar{\vartheta }_{k,n}-\vartheta_0 \right|\,\left|\bar{\vartheta }
_{N}-\vartheta_0 \right|\right)\\
&\quad + {\II\left(\bar{\vartheta }_{k,n} \right)^{-1}}\frac{1}{\sqrt{k}}
\sum_{j=N+1}^{k}\int_{\TT}^{}\dot{\ell }\left( \bar{\vartheta
}_{N} ,t\right)
\left[ {\rm d}X_{j}\left( t\right) -\lambda \left( \vartheta_{0} ,t
\right) {\rm d}t\right].
\end{align*}
We see that if we take $\frac{1}{2}-\gamma-\frac{\delta}{2} <0$ then
\begin{eqnarray*}
\sqrt{k}\left|\bar{\vartheta }_{k,n}-\vartheta_0 \right|\left|\bar{
\vartheta }_{N}-\vartheta_0 \right|&=& n^{\frac{1}{2}}n^{-\gamma}n^{-\frac{\delta}{2}
}\left(n^{\gamma}\left|\bar{\vartheta }_{k,n}-\vartheta_0 \right|
\right) n^{\frac{\delta}{2} }\left|\bar{\vartheta }_{N}-\vartheta_0
\right|\rightarrow 0
\end{eqnarray*}%
Therefore if $\delta \in \left( \frac{1}{3}, \frac{1}{2}\right) $, then we can take such $
\gamma$, that $\gamma <\delta$ and $\gamma>\frac{1-\delta}{2}$.
Finally we obtain
\begin{align*}
\sqrt{k}\left(\vartheta _{k,n}^{\star \star}-\vartheta_0 \right)&=
\frac{\II\left(\vartheta_{0} \right)^{-1}}{\sqrt{k}} \sum_{j=N+1}^{k}\int_{\TT}^{}
\dot{\ell }\left( \vartheta_{0} ,t\right)
\left[ {\rm d}X_{j}\left( t\right) -\lambda \left( \vartheta_{0} ,t
\right) {\rm d}t\right]+o\left(1\right) \\
&\qquad \qquad\qquad\qquad \Longrightarrow  \mathcal{N}\left( 0,\II\left(\vartheta_{0} \right)^{-1}\right).
\end{align*}%
Therefore we proved the following theorem
\begin{theorem}
\label{T4}
Let the conditions ${\cal P}_0$ and ${\cal L}_0$ be fulfilled. Then the
Two-step MLE-process $\left(\vartheta
_{k,n}^{\star \star} ,k=N+1,...,n\right) $ is asymptotically normal
\begin{equation*}
\sqrt{k}\left(\vartheta _{k,n}^{\ast \ast}-\vartheta_0 \right) \Longrightarrow
\mathcal{N}\left( 0,\II\left( \vartheta_0 \right) ^{-1}\right).
\end{equation*}
Here $k=\left[sn\right]$.
\end{theorem}

{\bf Example 4}. Suppose that the intensity function of the observed
inhomogeneous Poisson process is
\begin{align*}
\lambda \left(\vartheta ,t\right)=A\sin\left(2\pi t+\vartheta \right)-\lambda
_0,\qquad 0\leq t\leq 1
\end{align*}
where $\vartheta \in \Theta =\left(\alpha ,\beta \right)$, $0<\alpha <\beta
<2\pi $ and $A<\lambda _0$. Let us take $g\left(t\right)=\cos\left(2\pi t\right)$  and note that
\begin{align*}
M\left(\vartheta \right)=\int_{0}^{1}g\left(t\right) \lambda
\left(\vartheta,t \right){\rm d}t =\frac{A}{2}\cos\left(\vartheta
\right),\qquad \vartheta =\arccos \left(\frac{2M\left(\vartheta \right)}{A}\right).
\end{align*}
The MME is
\begin{align*}
\vartheta_n^*=\arccos\left(\frac{2}{An}\sum_{j=1}^{n}\int_{0}^{1}\cos\left(2\pi
t\right){\rm d}X_j\left(t\right)\right) .
\end{align*}
The Fisher information
$$
\II=\int_{0}^{1}\frac{A^2\cos^2\left(2\pi t\right)}{A\sin\left(2\pi t\right)+\lambda _0}{\rm d}t
$$
does not depend on $\vartheta $.
Let us take $N=\left[n^{\frac{4}{9}}\right]$ and introduce the Two-state
MLE-process as follows
\begin{align*}
\bar\vartheta _{k,n}&=\vartheta_N^*+\frac{1}{\II k
}\sum_{j=N+1}^{k}\int_{0}^{1}\frac{A\cos\left(2\pi t+
  \vartheta_N^*\right)}{A\sin\left(2\pi t+\vartheta_N^*\right)+\lambda
  _0}\; {\rm d}X_j\left(t\right),\quad k=N+1,\ldots,n,\\
\vartheta _{k,n}^{\ast \ast}&=\bar\vartheta _{k,n}+\frac{1}{\II k
}\sum_{j=N+1}^{k}\int_{0}^{1}\frac{A\cos\left(2\pi t+
  \vartheta_N^*\right)}{A\sin\left(2\pi t+\vartheta_N^*\right)+\lambda
  _0}\; {\rm d}X_j\left(t\right)\\
&\qquad -\frac{k-N}{\II k}\int_{0}^{1}\frac{\left[A\cos\left(2\pi t+
  \vartheta_N^*\right)\right]\left[ A\sin\left(2\pi t+\bar\vartheta _{k,n}\right)+\lambda
  _0 \right]}{A\sin\left(2\pi t+\vartheta_N^*\right)+\lambda
  _0}\; {\rm d}t
\end{align*}
because
\begin{align*}
\int_{0}^{1}A\cos\left(2\pi t+  \vartheta_N^*\right){\rm d}t=0 .
\end{align*}
By the Theorem \ref{T4}
\begin{align*}
\sqrt{k}\left(\vartheta _{k,n}^{\ast \ast}-\vartheta _0\right)\Longrightarrow
     {\cal N}\left(0, \II^{-1}\right).
\end{align*}

\section{Discussions}

It is clear that we can continue the process and to reduce the time of
learning using Three and more-step MLE (see \cite{Kut16}, where the
construction of the Thre-step MLE-process is discussed).

The space $\TT$ can
be of more general nature. For example, it can be $\RR^m$.

The similar construction  of one, two  and three-step MLE-processes in the case of
nonlinear time-series were realized in
the work \cite{KM16}. The numerical simulation of the two-step MLE-process
presented there show the good convergence of the estimation process to the
true value.

Note that the multi-step MLE-processes were used in the problems of
approximation of the solution of the Backward Stochastic Differential Equation
(see, e.g. \cite{Kut14}, \cite{K16}).

All these allow to think that the proposed construction of the  multi-step
MLE-processes is in some sense universal and can be used for the other models
of observations too.

\bigskip

{\bf Acknowledgement.}
This work was done under partial financial support of the grant of
RSF number 14-49-00079.

\end{document}